\documentclass{birkjour}
\usepackage{amsmath, amsthm, amsfonts, amssymb}
\usepackage{mathrsfs}
\usepackage{txfonts}

\newtheorem{theorem}{Theorem}[section]
\newtheorem{definition}{Definition}[section]
\newtheorem{coro}{Corollary}[section]
\newtheorem{lemma}{Lemma}[section]

\theoremstyle{definition}

\allowdisplaybreaks

\usepackage[numbers]{natbib}
\setcitestyle{open={},close={}}

\def \l {\left}
\def \r {\right}

\def \bR {\Bbb R}

\def \bB {\Bbb B}

\def\be{\begin{equation}}
\def\ee{\end{equation}}
\def\ba{\begin{array}}
	\def\ea{\end{array}}
\def\bd{\begin{definition}}
	\def\ed{\end{definition}}
\def\bt{\begin{theorem}}
	\def\et{\end{theorem}}
\def\bc{\begin{corollary}}
	\def\ec{\end{corollary}}
\def\bl{\begin{lemma}}
	\def\el{\end{lemma}}
\def\bdm{\begin{displaymath}}
\def\edm{\end{displaymath}}

\begin{document}

%
%
%
%
%
%
%
%
%

\title[sublinear operators on generalized mixed Morrey spaces]
 {Boundedness criterion for sublinear operators and commutators on generalized mixed Morrey spaces}

\author[M. Wei]{Mingquan Wei$^*$\footnote {$^*$ Corresponding author}}

\address{%
School of Mathematics and Stastics, Xinyang Normal University\\
Xinyang 464000, China}

\email{weimingquan11@mails.ucas.ac.cn}


\subjclass{Primary 42B20; Secondary 42B25, 42B35}

\keywords{Generalized mixed Morrey spaces, sublinear operator, commutators, ${\rm BMO}(\bR^n)$, Hardy-Littlewood maximal operator, fractional integral operator}

\date{September 26, 2021}

\begin{abstract}
In this paper, the author studies the boundedness for a large class of sublinear operators $T_\alpha, \alpha\in[0,n)$
generated by Calder{\'o}n-Zygmund operators  ($\alpha=0$) and generated by fractional integral operator ($\alpha>0$) on generalized mixed Morrey spaces $M^\varphi_{\vec{q}}(\bR^n)$. Moreover, the boundeness for commutators of  $T_\alpha, \alpha\in[0,n)$ on generalized mixed Morrey spaces $M^\varphi_{\vec{q}}(\bR^n)$ is also studied. As applications, we obtain the boundedness for Hardy-Littlewood maximal operator, Calder{\'o}n-Zygmund singular integral operators, fractional integral operator, fractional maximal operator and their commutators on generalzied mixed Morrey spaces.
\end{abstract}

\maketitle
\section{Introduction}

For $x\in\bR^n$, and $r>0$, let $B(x,r)$ be the open ball centered at $x$ with the radius $r$, and $B^ c(x,r)$ be its complement. The well-known fractional maximal operator $M_\alpha$ and fractional integral operator $I_\alpha$ are defined by 
$$M_\alpha f(x)=\sup_{ r>0}\frac{1}{|B(x,r)|^{1-\alpha/n}}\int_{B(x,r)}|f(y)|dy,~~0\leq\alpha< n,$$
and
$$I_\alpha f(x)=\int_{\bR^n}\frac{f(y)}{|x-y|^{n-\alpha}}dy,~~0<\alpha<n$$
respectively, where $f$ is locally integrable and $|B(x,r)|$ is the Lebesgue measure of $B(x,r)$. If $\alpha=0$, then $M\equiv M_0$ is the classical Hardy-Littlewood maximal operator.

Other than fractional maximal operator $M_\alpha$ and fractional integral operator $I_\alpha$,  Calder{\'o}n-Zygmund singular integral operators (see [\cite{grafakos2008classical}]) are also basic integral operators in harmonic analyis.
A Calder{\'o}n-Zygmund singular integral operator is a linear operator bounded from $L^2(\bR^n)$
to $L^2(\bR^n)$, which takes all infinitely continuously differentiable functions $f$ with
compact support to the functions $f\in L_{{\rm loc}}(\bR^n)$ and can be represented by
$$Kf(x)=\int_{\bR^n}k(x,y)f(y)dy, ~~ \mathrm{a.e.~off~supp}f.$$
Here $k(x,y)$ is a continuous function away from the diagonal which satisfies
the standard estimates: there exist some $0<\epsilon\leq1$, such that 
$$|k(x,y)|\lesssim \frac{1}{|x-y|^n}$$
for all $x,y\in \bR^n, x\neq y$ and
$$|k(x,y)-k(x',y)|+|k(y,x)-k(y,x')|\lesssim\frac{|x-x'|^\epsilon}{|x-y|^{n+\epsilon}}$$
whenever $2|x-x'|<|x-y|$.

In order to study the above operators and some related operators in harmonic analysis uniformly, many researchers introduced the following sublinear operators satisfying some size conditions.

The first one is $T\equiv T_0$, which is a sublinear opertor, and satisfies that for any $f\in L^1(\bR^n)$ with compact support and $x\notin \mathrm{supp}f$,
\begin{equation}\label{T-1}
|Tf(x)|\lesssim \int\frac{|f(y)|}{|x-y|^n}dy.
\end{equation}
Another one is $T_\alpha$ ($0<\alpha<n$), the fractional version of $T_0$, which is also a sublinear opertor, and satisfies that for any $f\in L^1(\bR^n)$ with compact support and $x\notin \mathrm{supp}f$,
\begin{equation}\label{T-2}
|T_\alpha f(x)|\lesssim \int\frac{|f(y)|}{|x-y|^{n-\alpha}}dy.
\end{equation}
We point out that condition (\ref{T-1}) was first introduced by Soria and
Weiss [\cite{soria1994remark}] and condition (\ref{T-2}) was introduced by Guliyev et al. [\cite{guliyev2011boundedness}]. Conditions (\ref{T-1}) and (\ref{T-2}) are satisfied by many interesting operators in harmonic analysis, such as the Calder{\'o}n-Zygmund singular integral operators, the Carleson's maximal operators, the Hardy-Littlewood maximal operators, the Fefferman's singular multipliers, the Fefferman's singular integrals, the Ricci-Stein's oscillatory singular integrals, the Bochner-Riesz means and so on (see
[\cite{lu2002singular,soria1994remark}] for more details).

As is well known, commutators are also important operators and play a key role in harmonic analysis. Recall that for a locally integrable function $b$ and an integral operator $T$, the commutator formed by $b$ and $T$ is defined by $[b,T]=bT-Tb$. Commutators of fractional maximal operator, fractional integral operator and Calder{\'o}n-Zygmund singular integral operators have been intensively studied, see [\cite{grafakos2008classical}] for more details. It is worthy pointing out that there are two different commutators of the fractional maximal operator $M_\alpha$. In this paper, the commutator of fractional maximal operator $[b,M_\alpha]$ under consideration is of the form $$[b,M_\alpha] f(x)=\sup_{ r>0}\frac{1}{|B(x,r)|^{1-\alpha/n}}\int_{B(x,r)}|b(x)-b(y)|~|f(y)|dy,~~0<\alpha\leq n,$$
for all locally integrable functions $f$ on $\bR^n$.

To study a class of commutators uniformly, one can also introduce some sublinear operators with additional size conditions as before. 
For a function $b$, suppose that the operator $T_b\equiv T_{b,0}$
represents a linear or a sublinear operator, which satisfies that for any $f\in L^1(\bR^n)$ with compact support and $x\notin \mathrm{supp}f$,
\begin{equation}\label{T-3}
|T_bf(x)|\lesssim \int\frac{|b(x)-b(y)|}{|x-y|^{n}}|f(y)|dy.
\end{equation}
Similarly, we assume that the operator $T_{b,\alpha}, \alpha\in(0,n)$ represents a linear or a sublinear operator, which
satisfies that for any $f\in L^1(\bR^n)$ with compact support and $x\notin \mathrm{supp}f$,
\begin{equation}\label{T-4}
|T_{b,\alpha}f(x)|\lesssim \int\frac{|b(x)-b(y)|}{|x-y|^{n-\alpha}}|f(y)|dy.
\end{equation}
The operator $T_{b,\alpha}, \alpha\in[0,n)$ has been studied in [\cite{guliyev2011boundedness,lu2002singular}].

As we know, the classical Morrey space is an important generalization of Lebesgue spaces. The classical Morrey space was introduced by Morrey in [\cite{morrey1938solutions}] to study the regularity of elliptic partial differential equations.
Nowadays, the Morrey space has become one of the most important function spaces in the theory of function spaces.
By a slight modification, Guliyev et al. [\cite{burenkov2009necessary,guliyev2011boundedness}] introduced  generalized Morrey spaces and studied the boundedness of many important operators in harmonic analysis on  generalized Morrey spaces. In particular, the boundedness of sublinear operators $T$, $T_{\alpha}$, $T_b$ and $T_{b,\alpha}$ was considered in [\cite{guliyev2011boundedness}]. Another interesting extension of the classical Morrey space is the mixed Morrey space, which was introduced by Nogayama et al. [\cite{Nogayama2019Boundedness,nogayama2019mixed,nogayama2020atomic}]. The boundedness of many integral operators and their commutators on mixed Morrey spaces was studied in [\cite{Nogayama2019Boundedness,nogayama2019mixed}].

Highly inspired by the work of Guliyev [\cite{burenkov2009necessary,guliyev2011boundedness,guliyev2013global}] and Nogayama et al. [\cite{Nogayama2019Boundedness,nogayama2019mixed,nogayama2020atomic}], we are going to study the boundedness of  the operators  $T$, $T_{\alpha}$, $T_b$ and $T_{b,\alpha}$ under some size conditions on  generalized mixed Morrey spaces in this paper. Here, generalized mixed Morrey spaces are the combination of generalized Morrey spaces and mixed Morrey spaces (see Definition \ref{GMM} in the following section), and theorefore are much more general. 
Our main results extend the boundedness of many operators on Morrey spaces [\cite{adams1975note,rosenthal2016boundedness,softova2013parabolic,yang2019existence}],  mixed Lebesgue spaces [\cite{1961The,chen2020iterated}],  generalized Morrey spaces [\cite{burenkov2009necessary,guliyev2011boundedness,guliyev2013global}], and mixed Morrey spaces [\cite{Nogayama2019Boundedness,nogayama2019mixed}]. Moreover, one can obtain the boundedness of many integral operators in harmonic analysis on generalized mixed Morrey spaces from our main theorems.

This paper is organized as follows. The definitions and some preliminaries are presented in Sect. 2. The boundedness of $T$ and $T_{\alpha}$ on generalized mixed Morrey spaces is studied in Sect. 3. The boundedness of $T_b$ and $T_{b,\alpha}$ on generalized mixed Morrey spaces is obtained in Sect. 4. Some applications are given in Sect. 5 to show the power of our main theorems.

\section{Definitions and preliminaries}

Throughout the paper, we use the following notations.

For any $r>0$ and $x\in \bR^n$, let $B(x,r)=\{y: |y-x|<r\}$ be the ball centered at $x$ with radius $r$.  Let $\bB=\{B(x,r):x\in\bR^n, r>0\}$ be the set of all such balls. We also use $\chi_E$ and $|E|$ to denote the characteristic function and the Lebesgue measure of a measurable set $E$.

The letter $\vec{p}$ denotes $n$-tuples of the numbers in $(0,\infty]$,~($n\geq1$),~$\vec{p}=(p_1,\cdots,p_n)$. By definition, the inequality, for example, $0<\vec{p}<\infty$ means $0<p_i<\infty$ for all $i$. For $1\leq\vec{p}\leq\infty$, we denote $\vec{p}'=(p'_1,\cdots,p'_n)$, where $p'_i$ satisfies
$\frac{1}{p_i}+\frac{1}{p'_i}=1$. By $A\lesssim B$, we
mean that $A\leq CB$ for some constant $C>0$, and $A\sim B$ means that $A\lesssim B$ and $B\lesssim A$.

Let $\mathcal{M}(\bR^n)$ be the class of Lebesgue measurable functions on $\bR^n$. 
For $0<\vec{p}<\infty$, a measurable function $f$ on $\bR^n$ belongs to $L^{\vec{p}}_{{\rm loc}}(\bR^n)$ if $f\chi_{E}\in L^{\vec{p}}(\bR^n)$ for any compact subset $E$ of $\bR^n$.
We also use $\mathbb{C}$ to represent all the complex numbers, and  $\mathbb{N}$ to represent the collection of all non-negative integers.

The classical Morrey space $M^p_q(\bR^n)$ is a natural generalization of Lebesgue spaces, which consist of all 
functions $f\in L^p_{{\rm loc}}(\bR^n)$ with finite norm
$$\|f\|_{M^p_q}=\sup_{x\in\bR^n,r>0}|B(x,r)|^{\frac{1}{p}-{\frac{1}{q}}}\|f\|_{L^q(B(x,r))},$$
where $1\leq q\leq p\leq\infty$.
Note that $M^p_q(\bR^n)=L^p(\bR^n)$ when $p=q$, and $M^p_q(\bR^n)=L^{\infty}(\bR^n)$ when $p=\infty$. If $q>p$, then $M^p_q(\bR^n)=\Theta$, where $\Theta$ is the set of all functions equivalent to 0 on $\bR^n$.

The classical Morrey space is a proper substitution when we consider the boundedness of integral operators in harmonic analysis. For example,
Chiarenza and Frasca [\cite{chiarenza1987morrey}] studied the boundedness of the maximal operator $M$ in $M^p_q(\bR^n)$. The well known Hardy-Littlewood-Sobolev inequality was also extented to Morrey spaces by  Spanne (but published by Peetre [\cite{peetre1969theory}]) and Adams [\cite{adams1975note}]. Since $M_\alpha f(x)\lesssim I_\alpha |f|(x), 0<\alpha<n$, the fractional maximal operator $M_\alpha$ is also bounded on $M^p_q(\bR^n)$.

In recent years, classical Morrey spaces $M^p_q(\bR^n)$ were extended to generalized Morrey spaces by Guliyev et al. [\cite{burenkov2009necessary,guliyev2021regularity,guliyev2011boundedness,guliyev2013global}].  

Let ~$1\leq q<\infty$ and $\varphi(x,r): \bR^n\times (0,\infty)\rightarrow(0,\infty)$ be a Lebesgue measurable function. A function $f\in\mathcal{M}(\bR^n)$ belongs to $M^\varphi_{q}(\bR^n)$, the generalized Morrey spaces, if it satisfies
\begin{eqnarray}
\|f\|_{M^\varphi_{q}}=\sup_{x\in\bR^n, r>0}\varphi(x,r)^{-1}|B(x,r)|^{-\frac{1}{q}}\|f\|_{L^q(B(x,r))}<\infty.
\end{eqnarray}
Roughly speaking, generalized Morrey space $M^\varphi_{q}(\bR^n)$ can be obtained by substituting $|B(x,r)|^{1/p-1/q}$ with a more general function in the definition of $M^p_q(\bR^n)$. From the definition, we recover the classical Morrey space $M^p_q(\bR^n)$ by taking $\varphi(x,r)=|B(x,r)|^{-1/p}$.
The boundedness of the sublinear operators $T$, $T_{\alpha}$, $T_b$ and $T_{b,\alpha}$ on spaces $M^\varphi_{q}(\bR^n)$ under some size conditions was obtained by Guliyev et al. [\cite{guliyev2011boundedness}].

In 2019, Nogayama [\cite{nogayama2019mixed}] considered a new Morrey type space, with the $L^p(\bR^n)$ norm replaced by the mixed Lebesgue norm $L^{\vec{q}}(\bR^n)$, which is call the mixed Morrey space.

We first recall the definition of mixed Lebesgue spaces defined in [\cite{1961The}].

Let ~$\vec{p}=(p_1,\cdots,p_n)\in(0,\infty]^n$. Then the mixed Lebesgue norm $\|\cdot\|_{L^{\vec{p}}}$ is defined by
\begin{eqnarray*}
	\|f\|_{L^{\vec{p}}}
	= \l(\int_{\bR}\cdots \l(\int_{\bR}\l(\int_{\bR}|f(x_1,x_2,\cdots,x_n)|^{p_1}dx_1\r)^{\frac{p_2}{p_1}}dx_2\r)^{\frac{p_3}{p_2}}\cdots dx_n\r)^{\frac{1}{p_n}}
\end{eqnarray*}
where $f: \bR^n \rightarrow \mathbb{C}$ is a measurable function. If $p_j=\infty$ for some $j=1,\cdots,n$, then we have to make appropriate modifications. We define the mixed Lebesgue space $L^{{\vec{p}}}(\bR^n)$
to be the set of all $f\in \mathcal{M}(\bR^n)$  with $\|f\|_{L^{\vec{p}}}<\infty$.

Let $1\leq \vec{q}<\infty, 1\leq p<\infty$ and $n/p\leq \sum_{i=1}^n1/{q_i}$. A function $f\in\mathcal{M}(\bR^n)$ belongs to the mixed Morrey spaces $M^p_{\vec{q}}(\bR^n)$ if 
$$\|f\|_{M^p_{\vec{q}}}=\sup_{x\in\bR^n,r>0}|B(x,r)|^{\frac{1}{p}-{\frac{1}{n}\l(\sum_{i=1}^n\frac{1}{q_i}\r)}}\|f\chi_{B(x,r)}\|_{L^{\vec{q}}}<\infty.$$
Obviously, we recover the classical Morrey space $M^p_{q}(\bR^n)$ when $\vec{q}=q$. We point out that in [\cite{Nogayama2019Boundedness,nogayama2019mixed}], the author used the cubes to define the mixed Morrey spaces. It is not hard to verify that the two definitions are equivalent.

In [\cite{nogayama2019mixed}], the Hardy-Littlewood maximal operator $M$, the fractional interal operator $I_\alpha$ and the singular integral operators $K$ were proved to be bounded in $M^p_{\vec{q}}(\bR^n)$. The boundedness of the commutator of $I_\alpha$ on $M^p_{\vec{q}}(\bR^n)$ was also obtained in [\cite{Nogayama2019Boundedness}].

Noting that generalized Morrey spaces and mixed Morrey spaces are different extensions of the classical Morrey spaces, it is natural for us to unify the two spaces. 

Now we are in a position to give the difinition of generalized mixed Morrey spaces.
\begin{definition}\label{GMM}
	Let $1\leq \vec{q}<\infty$, and $\varphi(x,r): \bR^n\times (0,\infty)\rightarrow(0,\infty)$ be a Lebesgue measurable function. A function $f\in\mathcal{M}(\bR^n)$ belongs to the mixed Morrey space $M^\varphi_{\vec{q}}(\bR^n)$ if 
	$$\|f\|_{M^\varphi_{\vec{q}}}=\sup_{x\in\bR^n,r>0}\varphi(x,r)^{-1}\|\chi_{B(x,r)}\|_{L^{\vec{q}}}^{-1}\|f\chi_{B(x,r)}\|_{L^{\vec{q}}}<\infty.$$	
\end{definition}	
Generalized mixed Morrey spaces contain generalized Morrey spaces and mixed Morrey spaces as special cases. In fact, $M^\varphi_{\vec{q}}(\bR^n)=M^\varphi_{q}(\bR^n)$ when $\vec{q}=q$, and $M^\varphi_{\vec{q}}(\bR^n)=M^p_{\vec{q}}(\bR^n)$ when $\varphi(x,r)=|B(x,r)|^{-1/p}$.

For a non-negative locally integrable function $w$ on $\bR^n$,  the weighted Hardy operators $H_w$ and $H^*_w$ are defined by defined 
$$H_wg(t):=\int_t^\infty g(s)w(s)ds,~~~0<t<\infty$$
and
$$H^*_wg(t):=\int_t^\infty \l(1+\ln\frac{s}{t}\r)g(s)w(s)ds,~~~0<t<\infty,$$
respectively, where $g\in L^1_{{\rm loc}}(\bR^n)$.

We have the following boundedness results for the weighted Hardy operators $H_w$ and $H^*_w$ proved in [\cite{guliev2012generalized,2013GENERALIZED}].
\begin{lemma}\label{L-Hardy}
	The inequality
	$$\mathrm{ess}\sup_{t>0}v_2(t)H_wg(t)\leq C~\mathrm{ess}\sup_{t>0}v_1(t)g(t)$$ 
	holds for all non-negative and non-increasing $g$ on $(0,\infty)$ if and only if
	$$A=\sup_{t>0}v_2(t)\int_t^\infty\frac{w(s)ds}{\mathrm{ess}\sup_{s<\tau<\infty}v_1(\tau)}<\infty,$$
	and $C\sim A$.
\end{lemma}
\begin{lemma}\label{L-Hardy-1}
	The inequality
	$$\mathrm{ess}\sup_{t>0}v_2(t)H^*_wg(t)\leq C~\mathrm{ess}\sup_{t>0}v_1(t)g(t)$$ 
	holds for all non-negative and non-increasing $g$ on $(0,\infty)$ if and only if
	$$A=\sup_{t>0}v_2(t)\int_t^\infty\l(1+\ln\frac{s}{t}\r)\frac{w(s)ds}{\mathrm{ess}\sup_{s<\tau<\infty}v_1(\tau)}<\infty,$$
	and $C\sim A$.
\end{lemma}



\section{Sublinear operators $T$ and $T_\alpha$ in spaces $M^\varphi_{\vec{q}}(\bR^n)$}
In this section, we investigate the boundedness of $T$ and $T_\alpha$ satisfying the size conditions (\ref{T-1}) and (\ref{T-2}) respectively, on generalized mixed Morrey spaces $M^\varphi_{\vec{q}}(\bR^n)$.

We first prove two lemmas, which give us the explicit estimats for the $L^{\vec{q}}(\bR^n)$ norm of $T$ and $T_\alpha$ on a given ball $B(x_0,r)$.
\begin{lemma}\label{L-T1}
	Let $1<\vec{q}<\infty$, $T$ be a sublinear operator satisfying condition {\rm (\ref{T-1})}, and bounded on $L^{\vec{q}}(\bR^n)$.
	
	Then for $1<\vec{q}<\infty$, the inequality 
	$$\|Tf\|_{L^{\vec{q}}(B(x_0,r))}\lesssim r^{\sum_{i=1}^n\frac{1}{q_i}}\int_{2r}^\infty t^{-1-\sum_{i=1}^n\frac{1}{q_i}}
	\|f\|_{L^{\vec{q}}(B(x_0,t))}dt$$
	holds for any ball $B(x_0,r)$ and all $f\in L^{\vec{q}}_{{\rm loc}}(\bR^n)$.
\end{lemma}
\begin{proof}
	For any ball $B=B(x_0,r)$, let $2B=B(x_0,2r)$ be the ball centered at $x_0$, with the radius $2r$. we represent $f$ as $f=f_1+f_2$, where
	$$f_1(y)=f\chi_{2B}(y),~~f_2(y)=f\chi_{(2B)^c}(y),~~r>0.$$
	Since $T$ is a sublinear operator, we have
	$$\|Tf\|_{L^{\vec{q}}(B)}\leq\|Tf_1\|_{L^{\vec{q}}(B)}+\|Tf_2\|_{L^{\vec{q}}(B)}.$$
	Noting that $f_1\in L^{\vec{q}}(\bR^n)$ and $T$ is bounded in $L^{\vec{q}}(\bR^n)$, we have
	$$\|Tf_1\|_{L^{\vec{q}}(B)}\leq\|Tf_1\|_{L^{\vec{q}}(\bR^n)}\lesssim\|f_1\|_{L^{\vec{q}}(\bR^n)}=
	\|f\|_{L^{\vec{q}}(2B)}.$$
	It is clear that $x\in B$, $y\in (2B)^c$ imply $\frac{1}{2}|x_0-y|\leq |x-y|\leq \frac{3}{2}|x_0-y|$, which further yields
	$$|Tf_2(x)|\lesssim\int_{(2B)^c}\frac{|f(y)|}{|x_0-y|^n}dy.$$	
	By Fubini's theorem, we have
	\begin{eqnarray*}
		\int_{(2B)^c}\frac{|f(y)|}{|x_0-y|^n}dy&\sim& \int_{(2B)^c}|f(y)|\int_{|x_0-y|}^\infty\frac{dt}{t^{n+1}}dy\\
		&\sim& \int_{2r}^\infty\int_{2r\leq |x_0-y|<t}|f(y)|dy\frac{dt}{t^{n+1}}\\
		&\lesssim& \int_{2r}^\infty\int_{B(x_0,t)}|f(y)|dy\frac{dt}{t^{n+1}}.
	\end{eqnarray*}
	Applying H{\"o}lder's inequality on mixed Lebesgue spaces (see [\cite{1961The}]), we obtain
	$$\int_{(2B)^c}\frac{|f(y)|}{|x_0-y|^n}dy\lesssim\int_{2r}^\infty\|f\|_{L^{\vec{q}}(B(x_0,t))}\frac{dt}{t^{1+\sum_{i=1}^n\frac{1}{q_i}}}.$$
	Moreover, for all $1<\vec{q}<\infty$, we have
	$$\|Tf_2\|_{L^{\vec{q}}(B(x_0,r))}\lesssim
	r^{\sum_{i=1}^n\frac{1}{q_i}}\int_{2r}^\infty\|f\|_{L^{\vec{q}}(B(x_0,t))}\frac{dt}{t^{1+\sum_{i=1}^n\frac{1}{q_i}}}.$$
	Therefore, one gets
	$$\|Tf\|_{L^{\vec{q}}(B(x_0,r))}\lesssim \|f\|_{L^{\vec{q}}(2B)}+
	r^{\sum_{i=1}^n\frac{1}{q_i}}\int_{2r}^\infty\|f\|_{L^{\vec{q}}(B(x_0,t))}\frac{dt}{t^{1+\sum_{i=1}^n\frac{1}{q_i}}}.$$
	On the other hand,
	\begin{eqnarray*}
		\|f\|_{L^{\vec{q}}(2B)}&\sim& r^{\sum_{i=1}^n\frac{1}{q_i}}
		\|f\|_{L^{\vec{q}}(2B)}\int_{2r}^\infty\frac{dt}{t^{1+\sum_{i=1}^n\frac{1}{q_i}}}\\
		&\lesssim& r^{\sum_{i=1}^n\frac{1}{q_i}}\int_{2r}^\infty\|f\|_{L^{\vec{q}}(B(x_0,t))}\frac{dt}{t^{1+\sum_{i=1}^n\frac{1}{q_i}}}.
	\end{eqnarray*}
	Thus
	$$\|Tf\|_{L^{\vec{q}}(B(x_0,r))}\lesssim 
	r^{\sum_{i=1}^n\frac{1}{q_i}}\int_{2r}^\infty\|f\|_{L^{\vec{q}}(B(x_0,t))}\frac{dt}{t^{1+\sum_{i=1}^n\frac{1}{q_i}}}.$$
\end{proof}
The second lemma is about the estimates of $\|T_\alpha f\|_{L^{\vec{p}}(B(x_0,r))}$.
\begin{lemma}\label{L-T2}
	Let $1<\vec{q}<\infty$, $0<\alpha<n/\vec{q}$, $\frac{1}{\vec{p}}=\frac{1}{\vec{q}}-\frac{\alpha}{n}$, $T_\alpha$ be a sublinear operator satisfying condition {\rm (\ref{T-2})}, and bounded from $L^{\vec{q}}(\bR^n)$ to $L^{\vec{p}}(\bR^n)$.
	
	Then for $1<\vec{q}<\infty$, the inequality 
	$$\|T_\alpha f\|_{L^{\vec{p}}(B(x_0,r))}\lesssim r^{\sum_{i=1}^n\frac{1}{p_i}}\int_{2r}^\infty t^{-1-\sum_{i=1}^n\frac{1}{p_i}}
	\|f\|_{L^{\vec{q}}(B(x_0,t))}dt$$
	holds for any ball $B(x_0,r)$ and all $f\in L^{\vec{q}}_{{\rm loc}}(\bR^n)$.
\end{lemma}
\begin{proof}
	For any ball $B=B(x_0,r)$, we also let $2B=B(x_0,2r)$ as before. Write $f=f_1+f_2$, with
	$f_1=f\chi_{2B}$ and $f_2=f\chi_{(2B)^c}$.
	
	By the fact that $T_\alpha$ is a sublinear operator, we have
	$$\|T_\alpha f\|_{L^{\vec{p}}(B)}\leq\|T_\alpha f_1\|_{L^{\vec{p}}(B)}+\|T_\alpha f_2\|_{L^{\vec{p}}(B)}.$$
	Noting that $f_1\in L^{\vec{q}}(\bR^n)$ and $T_\alpha$ is bounded from $L^{\vec{q}}(\bR^n)$ to $L^{\vec{p}}(\bR^n)$, we have
	$$\|T_\alpha f_1\|_{L^{\vec{p}}(B)}\leq\|T_\alpha f_1\|_{L^{\vec{p}}(\bR^n)}\lesssim\|f_1\|_{L^{\vec{q}}(\bR^n)}=
	\|f\|_{L^{\vec{q}}(2B)}.$$
	Since $x\in B$, $y\in (2B)^c$ imply $\frac{1}{2}|x_0-y|\leq |x-y|\leq \frac{3}{2}|x_0-y|$, it yields that
	$$|T_\alpha f_2(x)|\lesssim\int_{(2B)^c}\frac{|f(y)|}{|x_0-y|^{n-\alpha}}dy.$$	
	By Fubini's theorem, we have
	\begin{eqnarray*}
		\int_{(2B)^c}\frac{|f(y)|}{|x_0-y|^{n-\alpha}}dy&\sim& \int_{(2B)^c}|f(y)|\int_{|x_0-y|}^\infty\frac{dt}{t^{n+1-\alpha}}dy\\
		&\sim& \int_{2r}^\infty\int_{2r\leq |x_0-y|<t}|f(y)|dy\frac{dt}{t^{n+1-\alpha}}\\
		&\lesssim& \int_{2r}^\infty\int_{B(x_0,t)}|f(y)|dy\frac{dt}{t^{n+1-\alpha}}.
	\end{eqnarray*}
	Applying H{\"o}lder's inequality on mixed Lebesgue spaces, we obtain
	\begin{equation}\label{E-1}
	\int_{(2B)^c}\frac{|f(y)|}{|x_0-y|^{n-\alpha}}dy\lesssim\int_{2r}^\infty\|f\|_{L^{\vec{q}}(B(x_0,t))}\frac{dt}{t^{1+\sum_{i=1}^n\frac{1}{p_i}}}.
	\end{equation}
	Moreover, for all $1<\vec{q}<\infty$, we have
	$$\|T_\alpha f_2\|_{L^{\vec{p}}(B(x_0,r))}\lesssim
	r^{\sum_{i=1}^n\frac{1}{p_i}}\int_{2r}^\infty\|f\|_{L^{\vec{q}}(B(x_0,t))}\frac{dt}{t^{1+\sum_{i=1}^n\frac{1}{p_i}}}.$$
	Therefore, 
	$$\|T_\alpha f\|_{L^{\vec{p}}(B(x_0,r))}\lesssim \|f\|_{L^{\vec{q}}(2B)}+
	r^{\sum_{i=1}^n\frac{1}{p_i}}\int_{2r}^\infty\|f\|_{L^{\vec{q}}(B(x_0,t))}\frac{dt}{t^{1+\sum_{i=1}^n\frac{1}{p_i}}}.$$
	On the other hand,
	\begin{eqnarray}\label{E-2}
	\|f\|_{L^{\vec{q}}(2B)}&\sim& r^{\sum_{i=1}^n\frac{1}{p_i}}
	\|f\|_{L^{\vec{q}}(2B)}\int_{2r}^\infty\frac{dt}{t^{1+\sum_{i=1}^n\frac{1}{p_i}}}\nonumber\\
	&\lesssim& r^{\sum_{i=1}^n\frac{1}{p_i}}\int_{2r}^\infty\|f\|_{L^{\vec{q}}(B(x_0,t))}\frac{dt}{t^{1+\sum_{i=1}^n\frac{1}{p_i}}}.
	\end{eqnarray}
	Thus
	$$\|T_\alpha f\|_{L^{\vec{p}}(B(x_0,r))}\lesssim 
	r^{\sum_{i=1}^n\frac{1}{p_i}}\int_{2r}^\infty\|f\|_{L^{\vec{q}}(B(x_0,t))}\frac{dt}{t^{1+\sum_{i=1}^n\frac{1}{p_i}}}.$$
\end{proof}
Now we can present the first main result in this section.
\begin{theorem}\label{T-T1}
	Let $1<\vec{q}<\infty$, and $(\varphi_1,\varphi_2)$ satisfy the condition
	$$\int_r^\infty\frac{{\mathrm{ess}\inf}_{t<s<\infty}\varphi_1(x,s)s^{\sum_{i=1}^n\frac{1}{q_i}}}{t^{1+\sum_{i=1}^n\frac{1}{q_i}}}dt\lesssim \varphi_2(x,r).$$
	
	Suppose $T$ is a sublinear operator satisfying condition {\rm (\ref{T-1})} which is bounded on $L^{\vec{q}}(\bR^n)$.	
	Then for $1<\vec{q}<\infty$, the operator $T$ is bounded form $M^{\varphi_1}_{\vec{q}}(\bR^n)$ to $M^{\varphi_2}_{\vec{q}}(\bR^n)$. Moreover,
	$$\|Tf\|_{M^{\varphi_2}_{\vec{q}}}\lesssim\|f\|_{M^{\varphi_1}_{\vec{q}}}.$$ 
\end{theorem} 
\begin{proof}
	By Lemma \ref{L-T1} and Lemma \ref{L-Hardy} with $v_2(r)=\varphi_2(x,r)^{-1}$, $v_1(r)=\varphi_1(x,r)r^{-\sum_{i=1}^n\frac{1}{q_i}}$, $g(r)=\|f\|_{L^{\vec{q}}(B(x,r))}$ and $w(r)=r^{-1-\sum_{i=1}^n\frac{1}{q_i}}$, we have
	\begin{eqnarray*}
		\|Tf\|_{M^{\varphi_2}_{\vec{q}}}&\lesssim& \sup_{x\in\bR^n,r>0}\varphi_2(x,r)^{-1}\int_{r}^\infty t^{-1-\sum_{i=1}^n\frac{1}{q_i}}
		\|f\|_{L^{\vec{q}}(B(x,t))}dt\\
		&\lesssim& \sup_{x\in\bR^n,r>0}\varphi_1(x,r)^{-1}r^{-\sum_{i=1}^n\frac{1}{q_i}}
		\|f\|_{L^{\vec{q}}(B(x,t))}\sim\|f\|_{M^{\varphi_1}_{\vec{q}}}.
	\end{eqnarray*}
\end{proof}
By taking $\vec{q}=q$ in Theorem \ref{T-T1}, we recover the result of Guliyev et al. [\cite{guliyev2011boundedness}, Theorem 4.5], which gave the boundedness of $T$ on generalized Morrey spaces.

Similarly, for the sublinear operator $T_\alpha$ with the size condition (\ref{T-2}), we have the second theorem. 
\begin{theorem}\label{T-T2}
	Let $1<\vec{q}<\infty$, $0<\alpha<n/\vec{q}$, $\frac{1}{\vec{p}}=\frac{1}{\vec{q}}-\frac{\alpha}{n}$, and $(\varphi_1,\varphi_2)$ satisfy the condition
	$$\int_r^\infty\frac{{\mathrm{ess}\inf}_{t<s<\infty}\varphi_1(x,s)s^{\sum_{i=1}^n\frac{1}{q_i}}}{t^{1+\sum_{i=1}^n\frac{1}{p_i}}}dt\lesssim \varphi_2(x,r).$$
	Suppose $T_\alpha$ is a sublinear operator satisfying condition {\rm (\ref{T-2})} which is bounded from $L^{\vec{q}}(\bR^n)$ to $L^{\vec{p}}(\bR^n)$.
	Then for $1<\vec{q}<\infty$, the operator $T_\alpha$ is bounded form $M^{\varphi_1}_{\vec{q}}(\bR^n)$ to $M^{\varphi_2}_{\vec{p}}(\bR^n)$. Moreover,
	$$\|T_\alpha f\|_{M^{\varphi_2}_{\vec{p}}}\lesssim\|f\|_{M^{\varphi_1}_{\vec{q}}}.$$ 
\end{theorem}
\begin{proof}
	By Lemma \ref{L-T2} and Lemma \ref{L-Hardy} with $v_2(r)=\varphi_2(x,r)^{-1}$, $v_1(r)=\varphi_1(x,r)r^{-\sum_{i=1}^n\frac{1}{q_i}}$, $g(r)=\|f\|_{L^{\vec{q}}(B(x,r))}$ and $w(r)=r^{-1-\sum_{i=1}^n\frac{1}{p_i}}$, we have
	\begin{eqnarray*}
		\|T_\alpha f\|_{M^{\varphi_2}_{\vec{p}}}&\lesssim& \sup_{x\in\bR^n,r>0}\varphi_2(x,r)^{-1}\int_{2r}^\infty t^{-1-\sum_{i=1}^n\frac{1}{p_i}}
		\|f\|_{L^{\vec{q}}(B(x,t))}dt\\
		&\lesssim& \sup_{x\in\bR^n,r>0}\varphi_1(x,r)^{-1}r^{-\sum_{i=1}^n\frac{1}{q_i}}
		\|f\|_{L^{\vec{q}}(B(x,t))}\sim\|f\|_{M^{\varphi_1}_{\vec{q}}}.
	\end{eqnarray*}
\end{proof}
By taking $\vec{q}=q$ in Theorem \ref{T-T2}, we recover the result of [\cite{guliyev2011boundedness}, Theorem 5.4]. 

\section{Sublinear operators $T_b$ and $T_{b,\alpha}$ in spaces $M^\varphi_{\vec{q}}(\bR^n)$}
In this section, we investigate the boundedness of $T_\alpha$ and $T_{b,\alpha}$ satisfying the size conditions (\ref{T-3}) and (\ref{T-4}) respectively, on generalized mixed Morrey space $M^\varphi_{\vec{q}}(\bR^n)$.

First, we review the definition of ${\rm BMO}(\bR^n)$, the bounded mean oscillation space.
A function $f\in L_{{\rm loc}}(\bR^n)$ belongs to ${\rm BMO}(\bR^n)$ if 
\begin{equation}\label{BMO}
\|f\|_{{\rm BMO}}=
\sup_{x\in\bR^n,r>0}\frac{1}{|B(x,r)|}\int_{B(x,r)}|f(y)-f_{B(x,r)}|dy<\infty.
\end{equation}
If one regards two functions whose difference is a constant as one, then the
space ${\rm BMO}(\bR^n)$ is a Banach space with respect to norm $\|\cdot\|_{{\rm BMO}}$.
The John-Nirenberg ineuqalitiy for ${\rm BMO}(\bR^n)$ yields that for any $1<q<\infty$ and $f\in {\rm BMO}(\bR^n)$, the ${\rm BMO}(\bR^n)$ norm of $f$ is equivalent to 
$$\|f\|_{{\rm BMO}^q}=
\sup_{x\in\bR^n,r>0}\l(\frac{1}{|B(x,r)|}\int_{B(x,r)}|f(y)-f_{B(x,r)}|^qdy\r)^{\frac{1}{q}}.$$
Recall that for any $\vec{q}=(q_1,\cdots,q_n)\in (1,\infty)^n$, the John-Nirenberg inequality for mixed norm spaces [\cite{ho2018mixed}] shows that the ${\rm BMO}(\bR^n)$ norm of all $f\in {\rm BMO}(\bR^n)$ is also equivalent to 
\begin{equation}\label{mbmo}
\|f\|_{{\rm BMO}^{\vec{q}}}=
\sup_{x\in\bR^n,r>0}\frac{\|(f-f_{B(x,r)})\chi_{B(x,r)}\|_{L^{\vec{q}}}}{\|\chi_{B(x,r)}\|_{L^{\vec{q}}}}.
\end{equation}
The following property for ${\rm BMO}(\bR^n)$ functions is valid.
\begin{lemma}
	Let $f\in {\rm BMO}(\bR^n)$. Then for all $0<2r<t$, we have
	\begin{equation}\label{ln}
	|f_{B(x,r)}-f_{B(x,t)}|\lesssim \|f\|_{{\rm BMO}}\ln\frac{t}{r}.
	\end{equation}
\end{lemma}

We first prove two lemmas, which give us the explicit estimats for the $L^{\vec{q}}(\bR^n)$ norm of $T_b$ and $T_{b,\alpha}$ on a given ball $B(x_0,r)$.
\begin{lemma}\label{L-T3}
	Let $1<\vec{q}<\infty$, $b\in {\rm BMO}(\bR^n)$, $T_b$ be a sublinear operator satisfying condition {\rm (\ref{T-3})}, and bounded on $L^{\vec{q}}(\bR^n)$.
	
	Then for $1<\vec{q}<\infty$, the inequality 
	$$\|T_bf\|_{L^{\vec{q}}(B(x_0,r))}\lesssim r^{\sum_{i=1}^n\frac{1}{q_i}}\int_{2r}^\infty \l(1+\ln\frac{t}{r}\r)t^{-1-\sum_{i=1}^n\frac{1}{q_i}}
	\|f\|_{L^{\vec{q}}(B(x_0,t))}dt$$
	holds for any ball $B(x_0,r)$ and all $f\in L^{\vec{q}}_{{\rm loc}}(\bR^n)$.
\end{lemma}
\begin{proof}
	For any ball $B=B(x_0,r)$, let $2B=B(x_0,2r)$. Write $f$ as $f=f_1+f_2$, where
	$f_1=f\chi_{2B}$ and $f_2=f\chi_{(2B)^c}$.
	
	Since $T_\alpha$ is a sublinear operator, we have
	$$\|T_b f\|_{L^{\vec{q}}(B)}\leq\|T_b f_1\|_{L^{\vec{q}}(B)}+\|T_b f_2\|_{L^{\vec{q}}(B)}.$$
	Noting that $f_1\in L^{\vec{q}}(\bR^n)$ and $T_\alpha$ is bounded in $L^{\vec{q}}(\bR^n)$, we have
	$$\|T_b f_1\|_{L^{\vec{q}}(B)}\leq\|T_b f_1\|_{L^{\vec{q}}(\bR^n)}\lesssim\|f_1\|_{L^{\vec{q}}(\bR^n)}=
	\|f\|_{L^{\vec{q}}(2B)}.$$
	Since $x\in B$, $y\in (2B)^c$ imply $\frac{1}{2}|x_0-y|\leq |x-y|\leq \frac{3}{2}|x_0-y|$, we get
	\begin{eqnarray*}
		|T_bf_2(x)|&\lesssim&\int_{\bR^n}\frac{|b(x)-b(y)|}{|x-y|^n}|f(y)|dy\\
		&\sim&\int_{(2B)^c}\frac{|b(x)-b(y)|}{|x_0-y|^n}|f(y)|dy.
	\end{eqnarray*}	
	By using generalized Minkowski's inequality on mixed Lebesgue spaces, we have
	\begin{eqnarray*}
		\|T_b f_2\|_{L^{\vec{q}}(B))}
		&\lesssim&\l\|\int_{(2B)^c}\frac{|b(\cdot)-b(y)|}{|x_0-y|^n}|f(y)|dy\r\|_{L^{\vec{q}}(B(x_0,r))}\\
		&\lesssim&\l\|\int_{(2B)^c}\frac{|b_B-b(y)|}{|x_0-y|^n}|f(y)|dy\r\|_{L^{\vec{q}}(B(x_0,r))}\\
		&+&\l\|\int_{(2B)^c}\frac{|b(\cdot)-b_B|}{|x_0-y|^n}|f(y)|dy\r\|_{L^{\vec{q}}(B(x_0,r))}
		\\
		&=&I_1+I_2.
	\end{eqnarray*}
	For the term $I_1$, we have
	\begin{eqnarray*}
		I_1&\sim& r^{\sum_{i=1}^n\frac{1}{q_i}}\int_{(2B)^c}\frac{|b_B-b(y)|}{|x_0-y|^n}|f(y)|dy\\
		&\sim& r^{\sum_{i=1}^n\frac{1}{q_i}}\int_{(2B)^c}|b(y)-b_B||f(y)|\int_{|x_0-y|}^\infty\frac{dt}{t^{n+1}}dy\\
		&\sim& r^{\sum_{i=1}^n\frac{1}{q_i}}\int_{2r}^\infty\int_{2r\leq |x_0-y|<t}|b(y)-b_B||f(y)|dy\frac{dt}{t^{n+1}}\\
		&\lesssim& r^{\sum_{i=1}^n\frac{1}{q_i}}\int_{2r}^\infty\int_{B(x_0,t)}|b(y)-b_B||f(y)|dy\frac{dt}{t^{n+1}}.
	\end{eqnarray*}
	Applying H{\"o}lder's inequality and (\ref{mbmo}), (\ref{ln}), we get
	\begin{eqnarray*}
		I_1
		&\lesssim& r^{\sum_{i=1}^n\frac{1}{q_i}}\int_{2r}^\infty\int_{B(x_0,t)}|b(y)-b_{B(x_0,t)}||f(y)|dy\frac{dt}{t^{n+1}}\\
		&+&r^{\sum_{i=1}^n\frac{1}{q_i}}\int_{2r}^\infty\int_{B(x_0,t)}|b_{B(x_0,r)}-b_{B(x_0,t)}||f(y)|dy\frac{dt}{t^{n+1}}\\
		&\lesssim& r^{\sum_{i=1}^n\frac{1}{q_i}}\int_{2r}^\infty\int_{B(x_0,t)}\|(b(\cdot)-b_{B(x_0,t)})\chi_{B(x_0,t)}
		\|_{L^{\vec{q}'}}\|f\|_{L^{\vec{q}}(B(x_0,t))}dy\frac{dt}{t^{n+1}}\\
		&+&r^{\sum_{i=1}^n\frac{1}{q_i}}\int_{2r}^\infty|b_{B(x_0,r)}-b_{B(x_0,t)}|~\|f\|_{L^{\vec{q}}(B(x_0,t))}
		\frac{dt}{t^{1+\sum_{i=1}^n\frac{1}{q_i}}}\\
		&\lesssim& \|b\|_{{\rm BMO}}r^{\sum_{i=1}^n\frac{1}{q_i}}\int_{2r}^\infty\l(1+\ln\frac{t}{r}\r)\|f\|_{L^{\vec{q}}(B(x_0,t))}
		\frac{dt}{t^{1+\sum_{i=1}^n\frac{1}{q_i}}}.
	\end{eqnarray*}
	In order to estimate $I_2$, note that 
	$$I_2=\int_{(2B)^c}\frac{|f(y)|}{|x_0-y|^n}dy\times\l\|b(\cdot)-b_B\r\|_{L^{\vec{q}}(B(x_0,r))}.$$
	It follows from (\ref{mbmo}) that 
	$$I_2\lesssim\|b\|_{{\rm BMO}}r^{\sum_{i=1}^n\frac{1}{q_i}}\int_{(2B)^c}\frac{|f(y)|}{|x_0-y|^n}dy.$$
	Thus by (\ref{E-1}), we get
	\begin{eqnarray*}
		I_2&\lesssim& \|b\|_{{\rm BMO}}r^{\sum_{i=1}^n\frac{1}{q_i}}\int_{2r}^\infty\|f\|_{L^{\vec{q}}(B(x_0,t))}
		\frac{dt}{t^{1+\sum_{i=1}^n\frac{1}{q_i}}}.
	\end{eqnarray*}
	Summing up $I_1$ and $I_2$, we get
	$$\|T_bf_2\|_{L^{\vec{q}}(B)}\lesssim r^{\sum_{i=1}^n\frac{1}{q_i}}\int_{2r}^\infty \l(1+\ln\frac{t}{r}\r)t^{-1-\sum_{i=1}^n\frac{1}{q_i}}
	\|f\|_{L^{\vec{q}}(B(x_0,t))}dt.$$
	Therefore, by (\ref{E-2}), there holds
	\begin{eqnarray*}
		\|T_bf_2\|_{L^{\vec{q}}(B)}&\lesssim& \|f\|_{L^{\vec{q}}(2B)}
		+r^{\sum_{i=1}^n\frac{1}{q_i}}\int_{2r}^\infty \l(1+\ln\frac{t}{r}\r)t^{-1-\sum_{i=1}^n\frac{1}{q_i}}
		\|f\|_{L^{\vec{q}}(B(x_0,t))}dt\\
		&\lesssim&r^{\sum_{i=1}^n\frac{1}{q_i}}\int_{2r}^\infty \l(1+\ln\frac{t}{r}\r)t^{-1-\sum_{i=1}^n\frac{1}{q_i}}
		\|f\|_{L^{\vec{q}}(B(x_0,t))}dt.
	\end{eqnarray*}
	We are done.
\end{proof}

The second lemma in this section is about the estimates of $\|T_{b,\alpha} f\|_{L^{\vec{p}}(B(x_0,r))}$.
\begin{lemma}\label{L-T4}
	Let $1<\vec{q}<\infty$, $0<\alpha<n/\vec{q}$, $\frac{1}{\vec{p}}=\frac{1}{\vec{q}}-\frac{\alpha}{n}$, $b\in {\rm BMO}(\bR^n)$, $T_{b,\alpha}$ be a sublinear operator satisfying condition {\rm (\ref{T-4})}, and bounded from $L^{\vec{q}}(\bR^n)$ to $L^{\vec{p}}(\bR^n)$.
	
	Then for $1<\vec{q}<\infty$, the inequality 
	$$\|T_{b,\alpha}f\|_{L^{\vec{p}}(B(x_0,r))}\lesssim r^{\sum_{i=1}^n\frac{1}{p_i}}\int_{2r}^\infty \l(1+\ln\frac{t}{r}\r)t^{-1-\sum_{i=1}^n\frac{1}{p_i}}
	\|f\|_{L^{\vec{q}}(B(x_0,t))}dt$$
	holds for any ball $B(x_0,r)$ and all $f\in L^{\vec{q}}_{{\rm loc}}(\bR^n)$.
\end{lemma}
\begin{proof}
	For any ball $B=B(x_0,r)$, let $2B=B(x_0,2r)$. Write $f$ as $f=f_1+f_2$, where
	$f_1=f\chi_{2B}$ and $f_2=f\chi_{(2B)^c}$ as before.
	
	Since $T_{b,\alpha}$ is a sublinear operator, we have
	$$\|T_{b,\alpha} f\|_{L^{\vec{q}}(B)}\leq\|T_{b,\alpha} f_1\|_{L^{\vec{q}}(B)}+\|T_{b,\alpha} f_2\|_{L^{\vec{q}}(B)}.$$
	Noting that $f_1\in L^{\vec{q}}(\bR^n)$ and $T_{b,\alpha}$ is bounded from $L^{\vec{q}}(\bR^n)$ to $L^{\vec{p}}(\bR^n)$, we have
	$$\|T_{b,\alpha}f_1\|_{L^{\vec{q}}(B)}\leq\|T_{b,\alpha} f_1\|_{L^{\vec{q}}(\bR^n)}\lesssim\|f_1\|_{L^{\vec{q}}(\bR^n)}=
	\|f\|_{L^{\vec{q}}(2B)}.$$
	Since $x\in B$, $y\in (2B)^c$ imply $\frac{1}{2}|x_0-y|\leq |x-y|\leq \frac{3}{2}|x_0-y|$, we get
	\begin{eqnarray*}
		|T_{b,\alpha}f_2(x)|&\lesssim&\int_{\bR^n}\frac{|b(x)-b(y)|}{|x-y|^{n-\alpha}}|f(y)|dy\\
		&\sim&\int_{(2B)^c}\frac{|b(x)-b(y)|}{|x_0-y|^{n-\alpha}}|f(y)|dy.
	\end{eqnarray*}	
	By using generalized Minkowski's inequality, we have
	\begin{eqnarray*}
		\|T_{b,\alpha} f_2\|_{L^{\vec{p}}(B))}&\lesssim&\l\|\int_{(2B)^c}\frac{|b(\cdot)-b(y)|}{|x_0-y|^{n-\alpha}}|f(y)|dy\r\|_{L^{\vec{p}}(B(x_0,r))}\\
		&\lesssim&\l\|\int_{(2B)^c}\frac{|b_B-b(y)|}{|x_0-y|^{n-\alpha}}|f(y)|dy\r\|_{L^{\vec{p}}(B(x_0,r))}\\
		&+&\l\|\int_{(2B)^c}\frac{|b(\cdot)-b_B|}{|x_0-y|^{n-\alpha}}|f(y)|dy\r\|_{L^{\vec{p}}(B(x_0,r))}
		\\
		&=&I_1+I_2.
	\end{eqnarray*}
	For the term $I_1$, we have
	\begin{eqnarray*}
		I_1&\sim& r^{\sum_{i=1}^n\frac{1}{p_i}}\int_{(2B)^c}\frac{|b_B-b(y)|}{|x_0-y|^{n-\alpha}}|f(y)|dy\\
		&\sim& r^{\sum_{i=1}^n\frac{1}{p_i}}\int_{(2B)^c}|b(y)-b_B||f(y)|\int_{|x_0-y|}^\infty\frac{dt}{t^{n+1-\alpha}}dy\\
		&\sim& r^{\sum_{i=1}^n\frac{1}{p_i}}\int_{2r}^\infty\int_{2r\leq |x_0-y|<t}|b(y)-b_B||f(y)|dy\frac{dt}{t^{n+1-\alpha}}\\
		&\lesssim& r^{\sum_{i=1}^n\frac{1}{p_i}}\int_{2r}^\infty\int_{B(x_0,t)}|b(y)-b_B||f(y)|dy\frac{dt}{t^{n+1-\alpha}}.
	\end{eqnarray*}
	Applying H{\"o}lder's inequality and (\ref{mbmo}), (\ref{ln}), we get
	\begin{eqnarray*}
		I_1
		&\lesssim& r^{\sum_{i=1}^n\frac{1}{p_i}}\int_{2r}^\infty\int_{B(x_0,t)}|b(y)-b_{B(x_0,t)}||f(y)|dy\frac{dt}{t^{n+1-\alpha}}\\
		&+&r^{\sum_{i=1}^n\frac{1}{p_i}}\int_{2r}^\infty\int_{B(x_0,t)}|b_{B(x_0,r)}-b_{B(x_0,t)}||f(y)|dy\frac{dt}{t^{n+1-\alpha}}\\
		&\lesssim& r^{\sum_{i=1}^n\frac{1}{p_i}}\int_{2r}^\infty\int_{B(x_0,t)}\|(b(\cdot)-b_{B(x_0,t)})\chi_{B(x_0,t)}
		\|_{L^{\vec{q}'}}\|f\|_{L^{\vec{q}}(B(x_0,t))}dy\frac{dt}{t^{n+1-\alpha}}\\
		&+&r^{\sum_{i=1}^n\frac{1}{p_i}}\int_{2r}^\infty|b_{B(x_0,r)}-b_{B(x_0,t)}|~\|f\|_{L^{\vec{q}}(B(x_0,t))}
		\frac{dt}{t^{1-\alpha+\sum_{i=1}^n\frac{1}{q_i}}}\\
		&\lesssim& \|b\|_{{\rm BMO}}r^{\sum_{i=1}^n\frac{1}{p_i}}\int_{2r}^\infty\l(1+\ln\frac{t}{r}\r)\|f\|_{L^{\vec{q}}(B(x_0,t))}
		\frac{dt}{t^{1+\sum_{i=1}^n\frac{1}{p_i}}}.
	\end{eqnarray*}
	For the term $I_2$, note that 
	$$I_2=\int_{(2B)^c}\frac{|f(y)|}{|x_0-y|^{n-\alpha}}dy\times\l\|b(\cdot)-b_B\r\|_{L^{\vec{p}}(B(x_0,r))}.$$
	It follows from (\ref{mbmo}) that 
	$$I_2\lesssim\|b\|_{{\rm BMO}}r^{\sum_{i=1}^n\frac{1}{p_i}}\int_{(2B)^c}\frac{|f(y)|}{|x_0-y|^{n-\alpha}}dy.$$
	Thus by (\ref{E-1}), we get
	\begin{eqnarray*}
		I_2&\lesssim& \|b\|_{{\rm BMO}}r^{\sum_{i=1}^n\frac{1}{p_i}}\int_{2r}^\infty\|f\|_{L^{\vec{q}}(B(x_0,t))}
		\frac{dt}{t^{1+\sum_{i=1}^n\frac{1}{p_i}}}.
	\end{eqnarray*}
	Summing up $I_1$ and $I_2$, we get
	$$\|T_{b,\alpha}f_2\|_{L^{\vec{p}}(B)}\lesssim r^{\sum_{i=1}^n\frac{1}{p_i}}\int_{2r}^\infty \l(1+\ln\frac{t}{r}\r)t^{-1-\sum_{i=1}^n\frac{1}{p_i}}
	\|f\|_{L^{\vec{q}}(B(x_0,t))}dt.$$
	Therefore, by (\ref{E-2}), there holds
	\begin{eqnarray*}
		\|T_{b,\alpha}f\|_{L^{\vec{p}}(B)}&\lesssim& \|f\|_{L^{\vec{q}}(2B)}
		+r^{\sum_{i=1}^n\frac{1}{p_i}}\int_{2r}^\infty \l(1+\ln\frac{t}{r}\r)t^{-1-\sum_{i=1}^n\frac{1}{p_i}}
		\|f\|_{L^{\vec{q}}(B(x_0,t))}dt\\
		&\lesssim&r^{\sum_{i=1}^n\frac{1}{p_i}}\int_{2r}^\infty \l(1+\ln\frac{t}{r}\r)t^{-1-\sum_{i=1}^n\frac{1}{p_i}}
		\|f\|_{L^{\vec{q}}(B(x_0,t))}dt.
	\end{eqnarray*}
	We are done.
\end{proof}
Now we give the boundedness of $T_b$ and $T_{b,\alpha}$ on generalized mixed Morrey spaces.
\begin{theorem}\label{T-T3}
	Let $1<\vec{q}<\infty$, and $(\varphi_1,\varphi_2)$ satisfy the condition
	$$\int_r^\infty\l(1+\ln \frac{t}{r}\r)\frac{{\mathrm{ess}\inf}_{t<s<\infty}\varphi_1(x,s)s^{\sum_{i=1}^n\frac{1}{q_i}}}{t^{1+\sum_{i=1}^n\frac{1}{q_i}}}dt\lesssim \varphi_2(x,r).$$
	
	Suppose $b\in {\rm BMO}(\bR^n)$, and $T_b$ is a sublinear operator satisfying condition {\rm (\ref{T-3})}, and bounded on $L^{\vec{q}}(\bR^n)$.
	Then for $1<\vec{q}<\infty$, the operator $T_b$ is bounded form $M^{\varphi_1}_{\vec{q}}(\bR^n)$ to $M^{\varphi_2}_{\vec{q}}(\bR^n)$. Moreover,
	$$\|T_bf\|_{M^{\varphi_2}_{\vec{q}}}\lesssim\|b\|_{{\rm BMO}}\|f\|_{M^{\varphi_1}_{\vec{q}}}.$$ 
\end{theorem} 
\begin{proof}
	The proof of Theorem \ref{T-T3} follows by Lemma \ref{L-T3} and Lemma
	\ref{L-Hardy-1} in the same manner as in the proof of Theorem \ref{T-T1}.
\end{proof}
From Theorem \ref{T-T3}, on can recover the result of Guliyev et al. [\cite{guliyev2011boundedness}, Theorem 6.6] by taking $\vec{q}=q$.

Similarly, for the sublinear operator $T_{b,\alpha}$ with size condition (\ref{T-4}), we have the following theorem. 
\begin{theorem}\label{T-T4}
	Let $1<\vec{q}<\infty$, $0<\alpha<n/\vec{q}$, $\frac{1}{\vec{p}}=\frac{1}{\vec{q}}-\frac{\alpha}{n}$, and $(\varphi_1,\varphi_2)$ satisfy the condition
	$$\int_r^\infty\l(1+\ln \frac{t}{r}\r)\frac{{\mathrm{ess}\inf}_{t<s<\infty}\varphi_1(x,s)s^{\sum_{i=1}^n\frac{1}{q_i}}}{t^{1+\sum_{i=1}^n\frac{1}{p_i}}}dt\lesssim \varphi_2(x,r).$$
	Suppose $b\in {\rm BMO}(\bR^n)$, and $T_{b,\alpha}$ is a sublinear operator satisfying condition {\rm (\ref{T-4})}, and bounded from $L^{\vec{q}}(\bR^n)$ to $L^{\vec{p}}(\bR^n)$.
	Then for $1<\vec{q}<\infty$, the operator $T_{b,\alpha}$ is bounded from $M^{\varphi_1}_{\vec{q}}(\bR^n)$ to $M^{\varphi_2}_{\vec{p}}(\bR^n)$. Moreover,
	$$\|T_{b,\alpha}f\|_{M^{\varphi_2}_{\vec{p}}}\lesssim \|b\|_{{\rm BMO}}\|f\|_{M^{\varphi_1}_{\vec{q}}}.$$ 
\end{theorem}
\begin{proof}
	The statement of Theorem \ref{T-T4} follows by Lemma \ref{L-T4} and Lemma
	\ref{L-Hardy-1} in the same manner as in the proof of Theorem \ref{T-T2}.
\end{proof}
By taking $\vec{q}=q$ in Theorem \ref{T-T4}, we recover the result of Guliyev et al. [\cite{guliyev2011boundedness}, Theorem 7.4], which proved the boundedness of $T_{b,\alpha}$ on the generalized Morrey spaces.

\section{Some applications}
This section gives some applications of our main theorems. We will show that many important integral operators and commutators appearing in harmonic analysis satisfy the assumptions mentioned above. Therefore, we can obtain the boundedness of various operators on generalized mixed Morrey spaces by using our main results.

As stated in the introduction, the Hardy-Littlewood maximal operator $M$ and the Calder{\'o}n-Zygmund singular integral operator $K$
are all sublinear operators satisfying the condition (\ref{T-1}). One can also see that $[b,M]$ and $[b,K]$  are also subilinear operators and satisfiy the condition (\ref{T-3}).
Note that the fractional integral operator $I_\alpha$ is a linear operator satisfying the condition (\ref{T-2}), and the commutator $[b,I_\alpha]$ is also a linear operator satisfying (\ref{T-4}).
So in order to apply our main theorems to the mentioned operators, we need to show the boundedness of $M, K, [b,M], [b,K]$ on $L^{\vec{q}}(\bR^n)$, and the boundedness of $I_\alpha, [b,I_\alpha]$ from $L^{\vec{q}}(\bR^n)$ to $L^{\vec{p}}(\bR^n)$. 

As we know, the  Hardy-Littlewood maximal operator $M$ is bounded on $L^{\vec{q}}(\bR^n)$, $1<\vec{q}<\infty$ (see [\cite{nogayama2019mixed}]), but there is no complete boundedness results for some other operators on mixed Lebesgue spaces. To prove the boundedness of some important operators on mixed Lebesgue spaces in a uniform way, we will give the extrapolation theorems on mixed Lebesgue spaces, which have their own interest. 

The extrapolaton theory on mixed Lebesgue spaces relies on the classical $A_p$ weight (see [\cite{grafakos2008classical}]). 
\begin{definition}
	For $1<p<\infty$, a non-negative function $w\in L^1_{\rm loc}(\bR^n)$ is said to be an $A_p$ weight if
	\begin{eqnarray*}
		[w]_{A_p}=\sup_{B\in\mathbb{B}}\l(\frac{1}{|B|}\int_Bw(x)dx\r)\l(\frac{1}{|B|}\int_Bw(x)^{-\frac{p'}{p}}dx\r)^{\frac{p}{p'}}<\infty.
	\end{eqnarray*}
	A non-negative local integrable function $w$ is said to be an $A_1$ weight if
	$$\frac{1}{|B|}\int_Bw(y)dy\leq Cw(x),~~~ a.e. x\in B$$
	for some constant $C>0$. The infimum of all such $C$ is denoted by $[w]_{A_1}$. We denote $A_\infty$ by the union of all $A_p~(1\leq p<\infty)$ functions.
\end{definition}

We also need the boundedness of $M$ on mixed norm spaces $L^{\vec{q}}(\bR^n)$, see [\cite{nogayama2019mixed}].
\begin{lemma}\label{l0-A}
	For $1<\vec{q}<\infty$, there holds
	\begin{equation}\label{0-A}
	\|Mf\|_{L^{\vec{q}}}\lesssim\|f\|_{L^{\vec{q}}}.
	\end{equation}
\end{lemma}

By $\mathfrak{F}$, we mean a family of pair $(f,g)$ of non-negative measurable functions that are not identical to zero. For such a family $\mathfrak{F}$, $p>0$ and a weight $w\in A_q$, the expression
$$\int_{\bR^n}f(x)^{p}w(x)dx\lesssim \int_{\bR^n}g(x)^{p}w(x)dx,~~~(f,g)\in\mathfrak{F}$$
means that this inequality holds for all pair $(f,g)\in\mathfrak{F}$ if the left hand side is finite, and the implicated constant depends only on $p$ and $[w]_{A_q}$.

Now we give the extrapolation theorems on mixed Lebesgue spaces. The first one is the diagonal extrapolation theorem.
\begin{theorem}\label{D-ext}
	Let $0<q_0<\infty$ and $\vec{q}=(q_1,\cdots,q_n)\in(0,\infty)^n$. Let $f,g\in\mathcal{M}(\bR^n)$. Suppose for every $w\in A_1$, we have
	\begin{equation}\label{3-a}
	\int_{\bR^n}f(x)^{q_0}w(x)dx\lesssim \int_{\bR^n}g(x)^{q_0}w(x)dx,~~~(f,g)\in\mathfrak{F}.
	\end{equation}
	Then if $\vec{q}>q_0$, we have
	\begin{equation}\label{3-b}
	\|f\|_{L^{\vec{q}}}\lesssim \|g\|_{L^{\vec{q}}},~~~(f,g)\in\mathfrak{F}.
	\end{equation}
\end{theorem}
Our second main result in this section is the off-diagonal extrapolation on mixed Lebesgue spaces.
\begin{theorem}\label{main-2}
	Let $f,g\in\mathcal{M}$. Suppose that for some $p_0$ and $q_0$ with $0< q_0<p_0<\infty$ and every $w\in A_1$,
	\begin{equation}\label{3-X}
	\l(\int_{\bR^n}f(x)^{p_0}w(x)dx\r)^{\frac{1}{p_0}}\lesssim \l(\int_{\bR^n}g(x)^{q_0}w(x)^{\frac{q_0}{p_0}}dx\r)^{\frac{1}{q_0}},~~(f,g)\in\mathfrak{F}.
	\end{equation}
	Then for all $q_0<\vec{q}<\frac{p_0q_0}{p_0-q_0}$, and $\vec{p}$ satisfies $1/\vec{q}-1/\vec{p}=1/q_0-1/p_0$,
	we have
	\begin{equation}\label{5-1-X}
	\|f\|_{L^{\vec{p}}}\lesssim \|g\|_{L^{\vec{q}}},~~~(f,g)\in\mathfrak{F}.
	\end{equation}
\end{theorem}
\begin{proof}
	
	By the similarity, we only prove Theorem \ref{main-2}.
	
	We use the Rubio de Francia iteration algorithm presented in [\cite{cruz2011weights}].
	
	Let $\bar{\vec{q}}=\vec{q}/q_0$ and $\bar{\vec{p}}=\vec{p}/p_0$. By the assumptions and Lemma \ref{l0-A}, the
	maximal operator is bounded on $L^{\bar{\vec{p}}'}(\bR^n)$, so there exists a positive constant $B$ such
	that
	$$\|Mf\|_{L^{\bar{\vec{p}}'}}\leq B \|f\|_{L^{\bar{\vec{p}}'}}.$$
	For any non-negative function $h$, define a new operator $\mathfrak{R}h$ by
	$$\mathfrak{R}h(x)=\sum_{k=0}^\infty\frac{M^kh(x)}{2^kB^k},$$
	where for $k\geq1$, $M^k$ denotes $k$ iterations of the maximal operator,
	and $M^0$ is the identity operator.
	
	The operator $\mathfrak{R}$ satisfies
	\begin{equation}\label{3-d}
	h(x)\leq\mathfrak{R}h(x),
	\end{equation}
	\begin{equation}\label{3-e}
	\|\mathfrak{R}h\|_{L^{\bar{\vec{p}}'}}\leq 2\|h\|_{L^{\bar{\vec{p}}'}},
	\end{equation}
	\begin{equation}\label{3-f}
	\|\mathfrak{R}h\|_{A_1}\leq 2B.
	\end{equation}
	The inequality (\ref{3-d}) is straight-forward.
	
	Since
	\begin{eqnarray*}
		M(\mathfrak{R}h)\leq\sum_{k=0}^\infty\frac{M^{k+1}h}{2^kB^k}
		\leq 2B
		\sum_{k=1}^\infty\frac{M^kh}{2^kB^k}
		\leq 2B\mathfrak{R}h,
	\end{eqnarray*}
	the properties (\ref{3-e})  and (\ref{3-f})  are consequences of Lemma \ref{l0-A} and the definition of $A_1$.
	
	Since the dual of $L^{\bar{\vec{p}}}(\bR^n)$ is $L^{\bar{\vec{p}}'}(\bR^n)$, we get
	\begin{eqnarray}\label{3-h}
	\|f\|_{L^{\vec{p}}}^{p_0}&=&\|f^{p_0}\|_{L^{\bar{\vec{p}}}}\\ \nonumber
	&\lesssim & \sup\l\{ \int_{\bR^n}|f(x)|^{p_0}h(x)dx:~\|h\|_{L^{\bar{\vec{p}}'}}\leq1, h\geq0 \r\}.
	\end{eqnarray}
	By using H{\"o}lder's inequality on mixed Lebesgue spaces and (\ref{3-d}), we have
	\begin{eqnarray}\label{3-0}
	\int_{\bR^n}f(x)^{p_0}h(x)dx&\lesssim& \int_{\bR^n}f(x)^{p_0}\mathfrak{R}h(x)dx \nonumber\\
	&\lesssim & \|f^{p_0}\|_{L^{\bar{\vec{p}}}}\|h\|_{L^{\bar{\vec{p}}'}}<\infty.
	\end{eqnarray}
	In view of (\ref{3-d}) and $\mathfrak{R}h\in A_1$, we use (\ref{3-X}) with $w=\mathfrak{R}h(x)$ to obtain
	$$\int_{\bR^n}f(x)^{p_0}h(x)dx\lesssim \int_{\bR^n}f(x)^{p_0}\mathfrak{R}h(x)dx\lesssim\l(\int_{\bR^n}g(x)^{q_0}[\mathfrak{R}h(x)]^{q_0/p_0}dx\r)^{p_0/q_0}.$$
	Combining (\ref{3-e}) with (\ref{3-0}) and using H{\"o}lder's inequality on mixed Lebesgue spaces again, we arrive at
	\begin{eqnarray}\label{3-i}
	\int_{\bR^n}f(x)^{p_0}h(x)dx&\lesssim & \|g^{q_0}\|^{p_0/q_0}_{L^{\bar{\vec{q}}}}\|(\mathfrak Rh)^{q_0/p_0}\|^{p_0/q_0}_{L^{\bar{\vec{q}}'}}\\\nonumber
	&\sim &\|g\|^{p_0}_{L^{\vec{q}}}\|(\mathfrak Rh)^{q_0/p_0}\|^{p_0/q_0}_{L^{\bar{\vec{q}}'}}.
	\end{eqnarray}
	A direct calculation yields $q_0\bar{\vec{q}}'=p_0\bar{\vec{p}}'$. Therefore
	\begin{equation}\label{NN}
	\|(\mathfrak Rh)^{q_0/p_0}\|^{p_0/q_0}_{L^{\bar{\vec{q}}'}}=\|\mathfrak Rh\|_{L^{\bar{\vec{p}}'}}\lesssim\|h\|_{L^{\bar{\vec{p}}'}}.
	\end{equation}
	By taking the supremum over all $h\in L^{\bar{\vec{p}}'}(\bR^n)$ with $\|h\|_{L^{\bar{\vec{p}}'}}\leq1$ and $h\geq 0$, (\ref{3-h}), (\ref{3-i}) and (\ref{NN}) give us the desired conclusion (\ref{5-1-X}).		
\end{proof}
We point out that when $n=2$, there are other versions of the diagonal extrapolation theorem [\cite{ho2016strong}] and the off-diagonal extrapolation theorem [\cite{tan2020off}] on mixed Lebesgue spaces, which are different from Theorem \ref{D-ext} and Theorem \ref{main-2}. 

By the density of smooth functions with compact support $C^\infty_c(\bR^n)$ in the mixed Lebesgue space $L^{\vec{q}}(\bR^n)$, $1<\vec{q}<\infty$ (see [\cite{1961The}]), one can apply Theorem \ref{D-ext} and Theorem \ref{main-2} to the mapping property of some sublinear operators. 
\begin{theorem}\label{subli}
	Suppose $0<q_0<\vec{q}<\infty$ and 
	$T$ is a sublinear operator such that for every $w\in A_1$,
	\begin{equation*}
	\int_{\bR^n}|Tf(z)|^{q_0}w(z)dz\lesssim \int_{\bR^n}|f(z)|^{q_0}w(z)dz,~~f\in C^\infty_c(\bR^n).
	\end{equation*}
	Then $T$ can be extended to a bounded operator on $L^{\vec{q}}(\bR^n)$.
\end{theorem}
\begin{proof}
	By Theorem \ref{D-ext}, for any $f\in C^\infty_c(\bR^n)$, we have
	$$\|Tf\|_{L^{\vec{q}}}\lesssim \|f\|_{L^{\vec{q}}}.$$
	Since $T$ is a sublinear operator, we have $|T(f)-T(g)|\leq |T(f-g)|$, and hence, for any $f,g\in C^\infty_c(\bR^n)$, we have
	$$\|T(f)-T(g)\|_{L^{\vec{q}}}\leq \|T(f-g)\|_{L^{\vec{q}}}\lesssim\|f-g\|_{L^{\vec{q}}}.$$
	Since $C^\infty_c(\bR^n)$ is dense in $L^{\vec{q}}(\bR^n)$, the above inequalities guarantee that $T$ can be extended
	to be a bounded operator on $L^{\vec{q}}(\bR^n)$.
\end{proof}
The following corollary is a consequence of Theorem \ref{subli} and the weighted boundedness of the corresponding operators.
\begin{coro}\label{C-1}
	Let $1<\vec{q}<\infty$, $b\in {\rm BMO}(\bR^n)$, then $M, K, [b,M], [b,K]$ are all bounded on $L^{\vec{q}}(\bR^n)$.
\end{coro}
\begin{proof}
	It is well known that $M, K, [b,M], [b,K]$ are all sublinear operators, and bounded on $L^{q_0}_w(\bR^n)$ for arbitrary $1<q_0<\infty$ and $w\in A_{q_0}$ (see [\cite{grafakos2008classical}] for example). Since $A_1\subset A_{q_0}$, Theorem \ref{subli} implies that $M, K$, $[b,M], [b,K]$ are all bounded on $L^{\vec{q}}(\bR^n)$ for all $q_0<\vec{q}<\infty$.
	In view of the arbitrariness of $1<q_0<\infty$, $M$, $K, [b,M], [b,K]$ are also bounded on $L^{\vec{q}}(\bR^n)$ for all $1<\vec{q}<\infty$.
\end{proof}
Similarly, we have the following theorem.
\begin{theorem}\label{subilinear-2}
	Suppose that $0< q_0<p_0<\infty$, and 
	$T$ is a sublinear operator such that every $w\in A_1$,
	\begin{equation}\label{3-X-x}
	\l(\int_{\bR^n}|Tf(x)|^{p_0}w(x)dx\r)^{\frac{1}{p_0}}\lesssim \l(\int_{\bR^n}|f(x)|^{q_0}w(x)^{\frac{q_0}{p_0}}dx\r)^{\frac{1}{q_0}},~~f\in C^\infty_c(\bR^n).
	\end{equation}
	Then for all $q_0<\vec{q}<\frac{p_0q_0}{p_0-q_0}$, and $\vec{p}$ satisfies $1/\vec{q}-1/\vec{p}=1/q_0-1/p_0$,
	$T$ can be extended to a bounded operator from $L^{\vec{q}}(\bR^n)$ to $L^{\vec{p}}(\bR^n)$.
\end{theorem}
Since the proof of Theorem \ref{subilinear-2} is similar to that of Theorem \ref{subli}, we leave it to readers.

To apply Theorem \ref{subilinear-2} to fractional integral operator and its commutator, we need a different
class of weights: Suppose $0<\alpha<n$, $1<q<n/\alpha$ and $1/q-1/p=\alpha/n$, we say $w\in A_{q,p}(\bR^n)$ if for any ball $B\in\bB$,
\begin{eqnarray*}
	\frac{1}{|B|}\int_{B}w(x)dx\l(\frac{1}{|B|}\int_{B}w(x)^{-q'/p}dx \r)^{p/q'}<\infty.
\end{eqnarray*}
Note that this is equivalent to $w\in A_r$, where $r=1+p/q'$, so, in particular if $w\in A_1$, then $w\in A_{q,p}$. The boundedness of fractional integral operator $I_\alpha$ and its commutator $[b,I_\alpha]$ on weighted Lebesgue spaces was obtained in [\cite{lu2007singular,muckenhoupt1974norm}].
\begin{lemma}\label{wang}
	Let $0<\alpha<n$ and $1<q<p<\infty$ with $1/q-1/p=\alpha/n$. If
	$w\in A_{q,p}$ and $b\in {\rm BMO}(\bR^n)$, then we have
	\begin{eqnarray*}
		\l(\int_{\bR^n}|I_{\alpha}f(x)|^pw(x)dx\r)^{1/p}
		\lesssim\l(\int_{\bR^n}|f(x)|^qw(x)^{q/p}dx\r)^{1/q},
	\end{eqnarray*}
	\begin{eqnarray*}
		\l(\int_{\bR^n}|[b,I_{\alpha}]f(x)|^pw(x)dx\r)^{1/p}
		\lesssim\l(\int_{\bR^n}|f(x)|^qw(x)^{q/p}dx\r)^{1/q}.
	\end{eqnarray*}
\end{lemma}
These results are usually stated with the class $A_{q,p}$ defined slightly differently, with
$w$ replaced by $w^p$ (see [\cite{cruz2006boundedness,lu2007singular,muckenhoupt1974norm}] for example). Our formulation, though non-standard, is better for our purposes.

By Theorem \ref{subilinear-2} and Lemma \ref{wang}, we can get the following result.
\begin{coro}\label{C-2}
	Let $0<\alpha<n$, $1<\vec{q}<n/\alpha$, $\frac{1}{\vec{q}}-\frac{1}{\vec{p}}=\frac{\alpha}{n}$, and $b\in {\rm BMO}(\bR^n)$, then both $I_\alpha$ and $[b,I_\alpha]$ are bounded from $L^{\vec{q}}(\bR^n)$ to $L^{\vec{p}}(\bR^n)$.
\end{coro}
\begin{proof}
	The proof is just a repetition of the proof of Corollary \ref{C-1}, so we omit the details.
\end{proof}
We point out that the boundedness of $I_\alpha$ on mixed Lebesgue spaces has already proved in [\cite{1961The}] in a more general setting. However, our proof, relying on the extrapolation theory on mixed norm spaces, has its own interest.

From the statement at the beginning of this section, we can obtain the boundedness of $M, K$ on generalized mixed Morrey spaces, whose proof is just a combination of Theorem \ref{T-T1}
and Corollary \ref{C-1}.
\begin{theorem}\label{T-T5}
	Let $1<\vec{q}<\infty$, and $(\varphi_1,\varphi_2)$ satisfy the condition
	$$\int_r^\infty\frac{{\mathrm{ess}\inf}_{t<s<\infty}\varphi_1(x,s)s^{\sum_{i=1}^n\frac{1}{q_i}}}{t^{1+\sum_{i=1}^n\frac{1}{q_i}}}dt\lesssim \varphi_2(x,r).$$
	Then the Hardy-Littlewood maximal operator $M$ and the Calder{\'o}n-Zygmund singular integral operator $K$ are both bounded from $M^{\varphi_1}_{\vec{q}}(\bR^n)$ to $M^{\varphi_2}_{\vec{q}}(\bR^n)$. Moreover,
	$$\|Mf\|_{M^{\varphi_2}_{\vec{q}}}\lesssim \|f\|_{M^{\varphi_1}_{\vec{q}}},$$
	$$\|Kf\|_{M^{\varphi_2}_{\vec{q}}}\lesssim \|f\|_{M^{\varphi_1}_{\vec{q}}}.$$
\end{theorem} 
Similarly, we have
\begin{theorem}\label{T-T6}
	Let $1<\vec{q}<\infty$, and $(\varphi_1,\varphi_2)$ satisfy the condition
	$$\int_r^\infty\l(1+\ln \frac{t}{r}\r)\frac{{\mathrm{ess}\inf}_{t<s<\infty}\varphi_1(x,s)s^{\sum_{i=1}^n\frac{1}{q_i}}}{t^{1+\sum_{i=1}^n\frac{1}{q_i}}}dt\lesssim \varphi_2(x,r).$$ 
	If $b\in {\rm BMO}(\bR^n)$, then the commutator of the Hardy-Littlewood maximal operator $[b,M]$ and the commutator of the Calder{\'o}n-Zygmund singular integral operator $[b,K]$ are both bounded from $M^{\varphi_1}_{\vec{q}}(\bR^n)$ to $M^{\varphi_2}_{\vec{q}}(\bR^n)$. Moreover,
	$$\|[b,M]f\|_{M^{\varphi_2}_{\vec{q}}}\lesssim \|f\|_{M^{\varphi_1}_{\vec{q}}},$$
	$$\|[b,K]f\|_{M^{\varphi_2}_{\vec{q}}}\lesssim \|f\|_{M^{\varphi_1}_{\vec{q}}}.$$
\end{theorem} 
For the boundedness of $I_{\alpha}, [b,I_{\alpha}]$ on generalized mixed Morrey spaces, we have the following theorems.
\begin{theorem}\label{T-T7}
	Let $1<\vec{q}<\infty$, $0<\alpha<n/\vec{q}$, $\frac{1}{\vec{p}}=\frac{1}{\vec{q}}-\frac{\alpha}{n}$, and $(\varphi_1,\varphi_2)$ satisfy the condition
	$$\int_r^\infty\frac{{\mathrm{ess}\inf}_{t<s<\infty}\varphi_1(x,s)s^{\sum_{i=1}^n\frac{1}{q_i}}}{t^{1+\sum_{i=1}^n\frac{1}{p_i}}}dt\lesssim \varphi_2(x,r).$$
	Then the fractional integral operator $I_\alpha$ is bounded from $M^{\varphi_1}_{\vec{q}}(\bR^n)$ to $M^{\varphi_2}_{\vec{p}}(\bR^n)$. Moreover,
	$$\|I_\alpha f\|_{M^{\varphi_2}_{\vec{p}}}\lesssim \|f\|_{M^{\varphi_1}_{\vec{q}}}.$$ 
\end{theorem}
\begin{theorem}\label{T-T8}
	Let $1<\vec{q}<\infty$, $0<\alpha<n/\vec{q}$, $\frac{1}{\vec{p}}=\frac{1}{\vec{q}}-\frac{\alpha}{n}$, and $(\varphi_1,\varphi_2)$ satisfy the condition
	$$\int_r^\infty\l(1+\ln \frac{t}{r}\r)\frac{{\mathrm{ess}\inf}_{t<s<\infty}\varphi_1(x,s)s^{\sum_{i=1}^n\frac{1}{q_i}}}{t^{1+\sum_{i=1}^n\frac{1}{p_i}}}dt\lesssim \varphi_2(x,r).$$
	If $b\in {\rm BMO}(\bR^n)$, then the commutator of the fractional integral operator $[b,I_\alpha]$ is bounded from $M^{\varphi_1}_{\vec{q}}(\bR^n)$ to $M^{\varphi_2}_{\vec{p}}(\bR^n)$. Moreover,
	$$\|[b,I_\alpha]f\|_{M^{\varphi_2}_{\vec{p}}}\lesssim \|f\|_{M^{\varphi_1}_{\vec{q}}}.$$
\end{theorem}
We pointed out that the results in Theorem \ref{T-T7} and Theorem \ref{T-T8} remain true for fractional maximal operator and its commutator, sicnce $M_\alpha f\lesssim I_\alpha |f|$ and $[b,M_\alpha]f\lesssim [b,I_\alpha] |f|$ for all $f\in L^1_{{\rm loc}}(\bR^n)$.

\section*{Data availability statement}
Data sharing not applicable to this article as no datasets were generated or analysed during the current study.

\section*{Acknowledgement(s)}

Right after we completed the first version of the current paper we discovered that generalized mixed Morrey spaces were introduced independently by Zhang and Zhou in [\cite{zhang2021boundedness}] in different ways. The author would like to thank  the anonymous referee for many valuable comments and suggestions. This work was supported by the Natural Science Foundation of Henan Province (Grant Nos. 202300410338) and the Nanhu Scholar Program for Young Scholars of Xinyang Normal University.

\end{document}